\documentclass[11pt,leqno]{article}
\usepackage{graphicx, amsfonts, amsthm, amsxtra, verbatim, multicol}
\usepackage[mathscr]{euscript}
\begin{document}

\begin{center}
{\LARGE\bf 
Remarks on a theorem of Perron} 
\footnote{ 
Keywords: 
Fuchsian differential equations, exponents of a singularity
} \\

\vspace{.25in} {\large {\sc C\'esar Camacho, Hossein Movasati}} \\
Instituto de Matem\'atica Pura e Aplicada, IMPA, \\
Estrada Dona Castorina, 110,\\
22460-320, Rio de Janeiro, RJ, Brazil, \\ 
Email: {\tt camacho@impa.br, hossein@impa.br}
\end{center}
\newtheorem{theo}{Theorem}
\newtheorem{lem}{Lemma}
\newtheorem{exam}{Example}
\newtheorem{coro}{Corollary}
\newtheorem{defi}{Definition}
\newtheorem{axio}{I}

\newtheorem{prob}{Problem}
\newtheorem{lemm}{Lemma}
\newtheorem{prop}{Proposition}
\newtheorem{rem}{Remark}
\newtheorem{conj}{Conjecture}
\newtheorem{conv}{Convergence}

\def\podu{{\sf pd}}   
\def\per{{\sf pm}}      
\def\perr{{\sf q}}        
\def\perdo{{\cal K}}   
\def\sfl{{\mathrm F}} 
\def\sp{{\mathbb S}}  
 
\newcommand\diff[1]{\frac{d #1}{dz}} 
\def\End{{\rm End}}              
\def\hol{{\rm Hol}}
\def\sing{{\rm Sing}}            
\def\spec{{\rm Spec}}            
\def\cha{{\rm char}}             
\def\Gal{{\rm Gal}}              
\def\jacob{{\rm jacob}}          
\def\tjurina{{\rm tjurina}}      
\newcommand\Pn[1]{\mathbb{P}^{#1}}   
\def\Ff{\mathbb{F}}                  
\def\Z{\mathbb{Z}}                   
\def\Gm{\mathbb{G}_m}                 
\def\Q{\mathbb{Q}}                   
\def\C{\mathbb{C}}                   
\def\O{{\cal O}}                     
\def\as{\mathbb{U}}                  
\def\ring{{\mathsf R}}                         
\def\R{\mathbb{R}}                   
\def\N{\mathbb{N}}                   
\def\A{\mathbb{A}}                   
\def\uhp{{\mathbb H}}                
\newcommand\ep[1]{e^{\frac{2\pi i}{#1}}}
\newcommand\HH[2]{H^{#2}(#1)}        
\def\Mat{{\rm Mat}}              
\newcommand{\mat}[4]{
     \begin{pmatrix}
            #1 & #2 \\
            #3 & #4
       \end{pmatrix}
    }                                
\newcommand{\matt}[2]{
     \begin{pmatrix}                 
            #1   \\
            #2
       \end{pmatrix}
    }
\def\ker{{\rm ker}}              
\def\cl{{\rm cl}}                
\def\dR{{\rm dR}}                

\def\hc{{\mathsf H}}                 
\def\Hb{{\cal H}}                    
\def\GL{{\rm GL}}                
\def\pese{{\sf P}}                  
\def\pedo{{\cal  P}}                  
\def\PP{\tilde{\cal P}}              
\def\cm {{\cal C}}                   
\def\K{{\mathbb K}}                  
\def\k{{\mathsf k}}                  
\def\F{{\cal F}}                     
\def\M{{\cal M}}
\def\RR{{\cal R}}
\newcommand\Hi[1]{\mathbb{P}^{#1}_\infty}
\def\pt{\mathbb{C}[t]}               
\def\W{{\cal W}}                     
\def\gr{{\rm Gr}}                
\def\Im{{\rm Im}}                
\def\Re{{\rm Re}}                
\def\depth{{\rm depth}}
\newcommand\SL[2]{{\rm SL}(#1, #2)}    
\newcommand\PSL[2]{{\rm PSL}(#1, #2)}  
\def\Resi{{\rm Resi}}              

\def\L{{\cal L}}                     
\def\Aut{{\rm Aut}}              
\def\any{R}                          
\newcommand\ovl[1]{\overline{#1}}    

\def\T{{\cal T }}                    
\def\tr{{\mathsf t}}                 
\newcommand\mf[2]{{M}^{#1}_{#2}}     
\newcommand\bn[2]{\binom{#1}{#2}}    
\def\ja{{\rm j}}                 
\def\Sc{\mathsf{S}}                  
\newcommand\es[1]{g_{#1}}            
\newcommand\V{{\mathsf V}}           
\newcommand\WW{{\mathsf W}}          
\newcommand\Ss{{\cal O}}             
\def\rank{{\rm rank}}                
\def\Dif{{\cal D}}                   
\def\gcd{{\rm gcd}}                  
\def\zedi{{\rm ZD}}                  
\def\BM{{\mathsf H}}                 
\def\plf{{\sf pl}}                             
\def\sgn{{\rm sgn}}                      
\def\diag{{\rm diag}}                   
\def\hodge{{\rm Hodge}}
\def\HF{{\sf F}}                                
\def\WF{{\sf W}}                               
\def\HV{{\sf HV}}                                
\def\pol{{\rm pole}}                               
\def\bafi{{\sf r}}
\def\codim{{\rm codim}}                               
\def\id{{\rm id}}                               
\def\gms{{\sf M}}                           
\def\Iso{{\rm Iso}}                           

\theoremstyle{plain}
\def\NN{{\Bbb N}}
\def\CC{{\Bbb C}}
\def\ZZ{{\Bbb Z}}
\def\QQ{{\Bbb Q}}
\def\SL{{\rm SL}}
\def\GO{{\rm GO}}
\def\GL{{\rm GL}}
\def\PGL{{\rm PGL}}
\def\dz{{\rm d} z}
\def\dx{{\rm d} x}

\def\la{{\lambda}}
\def\nlambda{{\tilde{\lambda}}}
\def\nue{{\nu}}
\def\ele{{\rm L}}
\def\tparam{{t}}
\def\Mu{{\mu}}
\def\th{{\theta}}
\def\thh{{\tilde{\theta}}}
\def\Th{{\Theta}}
\def\tr{{\rm tr}}
\def\rk{{\rm rk}}
\def\Trace{{\rm Tr}}
\def\Mat{{\rm Mat}}
\def\al{{\alpha}}
\def\diag{{\rm diag}}
\def\id{{\rm id}}
\def\k{{\frak k}}
\def\l{{\frak  l}}
\def\Pn{\mathbb P}

\begin{abstract}
For a linear differential equation with a mild  condition on its singularities, we discuss generalized continued fractions converging to expressions in its solutions and their derivatives. In the case of an order two linear differential equation, this is  the logarithmic derivative of the holomorphic solution near a singularity.   
\end{abstract}

\section{Introduction}
In \cite{ince} Ince analyzed  the process introduced by Perron \cite{perron} leading to find solutions of  the Gauss hypergeometric equation:
\begin{equation}
\label{gauss}
z(1-z)y''+(c-(a+b+1)z)y'-aby=0,\  \ a,b,c\in\C-\Z.
\end{equation}
by means of successive differentiation of (\ref{gauss}). This process is equivalent to the set of recurrence relations
\begin{equation}
\label{recrel}
x_n=a_{1,n}x_{n+1}+a_{2,n}x_{n+2}, 
 \text{  where } x_n=\frac{y^{(n)}}{n!},\ n=0,1,2,\ldots 
\end{equation}
and
$$
a_{1,n}:=\frac{(n+1)(c+n-(a+b+2n+1)z)}{(a+n)(b+n)},
$$
$$
a_{2,n}:=\frac{(n+1)(n+2)z(z-1)}{(a+n)(b+n)},\ n=1,2,\ldots
$$
The corresponding continued fraction
\begin{equation}
\label{cfr}
\cfrac{1}{a_{1,0}+\cfrac{a_{2,0}}{a_{1,1}+\cfrac{a_{2,1}}{\ddots\cfrac{\ddots}{a_{1,n-3}+\cfrac{a_{2,n-3}}{a_{1,n-2}}}}}}
\end{equation}
converges on one hand  to $(\ln F(a,b,c|z))'$ in the region $\Re(z)<\frac{1}{2}$ and on the other hand to 
$(\ln F(a,b,a+b-c+1|1-z))'$ in the region $\Re(z)>\frac{1}{2}$. In this paper we present a generalization of this process to second order Fuchsian differential equations using Poincar\'e's theorem on recurrence sequences. We would like to thank J. Oesterl\'e for many useful conversations and to Frits Beukers who indicated to us the relevance of  the  Poincar\'e's recurrence theorem for this matter.
\section{Poincar\'e's theorem on recurrent sequences}
In \cite{poi} Section 2 pages 213-217 and Section 6 page 237, 
Poincar\'e proved the following theorem.
\begin{theo}
For a fixed $k\in \N$, let us be given a recurrent sequence $f_n$
\begin{equation}
\label{recrel}
f_{n+k}+a_{k-1,n}f_{n+k-1}+\cdots+ a_{1,n}f_{n+1}+a_{0,n}f_n=0, 
\ n\in\N_0,
\end{equation}
where $\lim_{n\to \infty}a_{j,n}=a_j<\infty$. Assume that the roots of the characteristic equation $z^k+\sum_{j=0}^{k-1}a_j z^j$ have different absolute values. Then either $f_n=0$ for large $n$ or $\lim \frac{f_{n+1}}{f_{n}}$ converges to a root of the characteristic equation. 
\end{theo}
 
Let $f=\sum_{n=0}^\infty f_nz^n$ be a holomorphic 
solution of the linear differential equation $L(f)=0$, where 
\begin{eqnarray}
\label{21july2013}
 L & := & P_0(\delta)+zP_1(\delta)+\cdots+z^k P_k(\delta) \\
  &=&  Q_0(z)\delta^{m}+Q_1(z)\delta^{m-1}+\cdots+ Q_m(z)
\end{eqnarray}
where $\delta=z\frac{d}{dz}$.  Comparing the coefficients of $z^{n+k}$ in both sides of $L(f)=0$ we get the  recurrence relation (\ref{recrel}) with  
$$
a_{j,n}=\frac{P_{k-j}(n+j)}{P_0(n+k)},\ j=0,1,\ldots,k-1.
$$
(for $z^j,\ 0\leq j<k$, we get restrictions on the coefficients $f_0,f_1,\ldots,f_{k-1}$). 
Any linear differential equation with a regular singularity at $z=0$ can be written 
in the format (\ref{21july2013}) with $\deg(P_j)\leq \deg(P_0),\ \ j=1,2,\ldots,k$, and so for these differential equations $\lim_{n\to \infty }a_{j,n}$ exists. It turns out that $Q_0$ is the characteristic polynomial of the recurrence relation and $Q_0(0)\not=0$. Note that $P_0$ is the indicial equation of $L$ at $z=0$ and for all except a finite number of points we have $P_0(\delta)=\delta(\delta-1)\cdots(\delta-(n-1))$.  
\begin{defi}\rm
For a linear differential equation $L$, we denote by $A_L$ the set of real lines in $\C$ perpendicular in the middle to the segments connecting the singularities of $L$.
\end{defi}
Now, we are able to reformulate Poincar\'e's theorem in the following format. 
\begin{theo}
\label{23july2013}
Let $L$ be a linear differential equation with regular singularities $\{t_1,t_2,\ldots,t_r\}$ at $\C$. 
 For any holomorphic non-polynomial solution $f=\sum_{n=0}^\infty f_n(z-z_0)^n$ around a point $z_0\in\C\backslash A_L$, the limit $\lim_{n\to \infty }\frac{f_{n+1}}{f_n}$ exists and it is one of $t_j$'s.  
\end{theo}
Note that in the above theorem $\infty$ may be an irregular singularity.  The mentioned $t_j$ is a singularity  such that $f$ converges in a disc with center $z_0$ and radius $|t_j-z_0|$ and not beyond this disc. It may contain some singularities of $L$.   
Finally note also that if 
$\lim_{n\to \infty }\frac{f_{n+1}}{f_n}=t_j$ exists then $\lim_{n\to \infty}f_n^{\frac{1}{n}}=t_j$.
 For further
references on Poincar\'e's theorem see \cite{mate}  

\section{A convergence theorem}
Let us be given a linear differential equation 
\begin{equation}
\label{lequation}
L:   \ \ y^{(m)}=\sum_{i=0}^{m-1} q_iy^{(i)},\ q_0,q_1,\ldots,q_{m-1}\in\C(z),
\end{equation}
where $\C(z)$ is the field of rational functions in $z$ with complex coefficients,
and let $S=\{t_1,t_2,\ldots,t_r\}\subset \C$ be the set of its 
singular points.   
We associate to $L$ the open sets $U_1,U_2,\ldots,U_r$ defined by 
\begin{equation}
\label{20.1.09}
U_i:=\{z\in \C\mid |z-t_i|< |z-t_j|,\ j=1,2,\ldots,r, \ j\not=i\}
\end{equation}
The set $\C\backslash \cup_{i=1}^rU_i$ consists of segment of lines perpendicular in the middle to the 
segments connecting two points $t_{j_1}$ and $t_{j_2},\ \ j_1,j_2=1,2,\ldots,r,\ \ j_1\not= j_2$. 

Note that $t_i$ is the unique element  of $S$  inside $U_i$. 

\begin{defi}\rm
\label{genericL}
We say that $L$ is generic if for each singularity $t_i$ of 
(\ref{lequation})  in the affine plane $\C$ there is a basis  
$y_1,y_2,\ldots,y_m$ of  the $\C$-vector space of its solutions such that  near $t_i$ the solutions $y_2,\ldots, y_m$ extend holomorphically to $t_i$ and the solution $y_1$ does not extends. 
\end{defi}
The reader may have noticed that we do not put any condition on 
the singularity $\infty$.

\begin{prob}
\label{icaperu}
Let $f_1,f_2$ be two solutions of a generic Fuchsian differential equation $L$ in $U_i$. Then the limit 
\begin{equation}
\label{15aug2013}
 \lim_{n\to \infty}\frac{f_1^{(n)}(z)}{ f_2^{(n)}(z)},\ \ z\in U_i
\end{equation}
exists and it is a constant number depending on $i$. 
\end{prob}

The following Proposition gives a partial answer to this problem.

\begin{prop}
\label{icaperu1}
Let $f_1,f_2$ be two solutions of a generic Fuchsian differential equation $L$. Then in each connected component $U$ of $\C\backslash A_L$,
the limit 
\begin{equation}
\label{15aug2013-1}
 \lim_{n\to \infty}\frac{f_1^{(n)}(z)}{ f_2^{(n)}(z)},\
\end{equation}
exists and it is a constant number depending on $U$. 
\end{prop}
\begin{proof}
Take $y_1,y_2,\ldots,y_m$ a basis of solutions of (\ref{lequation}) associated to $t_i$ as it is described in Definition  \ref{genericL}.
For $z\in\C\backslash S$,  let $\rho_j(z)$ be the maximum real number such that $y_j$ is holomorphic in  a disc with center  
$z$ and radius $\rho_j(z)$.  Indeed $\rho_j(z)$ is the convergence radius of the Taylor series of $y_j$ at $z$.
 By Cauchy-Hadamard theorem and Poincar\'e's theorem  on recurrence relations,  for $z\in \C \backslash A_L$ we have 
$$
\rho_j(z)=(\lim_{n}(\frac{|y_j^{(n)}(z)|}{n!})^{\frac{1}{n}})^{-1}
$$
($\overline{\lim}$ is substituted by $\lim$ which is a stronger statement.)
Consider  $U_i,\ i=1,2,\ldots,r$.
Since $y_1$ (resp. $y_j,\ j\not=1$) is not (resp.  is) holomorphic at $t_i$  and $|z-t_j|>|z-t_i|$, we have 
$\rho_1(z)=|z-t_i|$ and $\rho_j(z)>|z-t_i|$. It follows that 
$$
\lim_{n}(\frac{|y_j^{(n)}(z)|} {|y_1^{(n)}(z)|})^{\frac{1}{n}}=\frac{\rho_1(z)}{\rho_j(z)}<1
$$
which  implies that
\begin{equation}
\label{res.tai}
\lim_{n\to\infty}(\frac{y_j^{(n)}(z)} {y_1^{(n)}(z)})=0.
\end{equation}
Since $f_1$ and $f_2$ are linear combinations of $y_j$'s the result follows. 
\end{proof}
We calculate the $n$-th derivative of $y$  
$$
y^{(n)}=\sum_{i=0}^{m-1} q_{i,n}y^{(i)}, 
$$
$$
q_{i,n}\in\C(z),\ i=0,1,\ldots,m-1,  n=m, m+1,\cdots
$$

\begin{theo}
\label{charla3}
For a generic Fuchsian differential equation (\ref{lequation}), the fraction  
$\frac{q_{k,n}}{q_{j,n}},\ k,j=0,1,\ldots,m-1,\ k\not= j$, converges uniformly in compact subsets of
 each component of $\C\backslash A_L$ as $n$ goes to infinity.
\end{theo}
If Problem \ref{icaperu} is solved then the convergence in the above theorem will occur in each $U_i$. 
\begin{proof}
Let $y_1,y_2,\ldots,y_m$ be a basis of the $\C$-vector space of solutions of (\ref{lequation}) and let
$$
Y^{(i)}=[y_1^{(i)},y_2^{(i)},\cdots,y_m^{(i)}]^{\tr}, \ i=1,2,3,\ldots 
$$
$$
Q_n=[q_{0,n},q_{1,n},\ldots,q_{m-1, n}]^{\tr}
$$
We have 
$$
Y^{(n)}=[Y,Y',\cdots,Y^{(m-1)}]Q_n,
$$
Therefore
\begin{equation}
\label{4oct08}
Q_n=[Y,Y',\cdots,Y^{(m-1)}]^{-1}Y^{(n)}
\end{equation}
Now, the theorem follows from Proposition \ref{icaperu}.
\end{proof}

\begin{rem}\rm
The matrix $X:=[Y,Y',\ldots, Y^{(m-1)}]^{\tr}$ is a fundamental system of the linear differential equation $X'=AX$, where 
$$
A=\begin{pmatrix}
   0&1&0&\cdots&0\\
   0&0&1&\cdots&0\\
   \vdots&\vdots &\vdots&\cdots&\vdots\\
   q_0&q_1&q_2&\cdots&q_{m-1}\\
  \end{pmatrix}
$$
and so $\tilde X=[Y,Y',\ldots, Y^{(m-1)}]^{-1}$ is a fundamental system of $\tilde X'=(-A^{\tr})\tilde X$.
\end{rem}

\section{Second order linear differential equations}
 Consider a linear differential equation of order two
\begin{equation}
\label{lequation1}
y^{''}=q_0(z)y+q_1(z)y', q_0,q_1\in\C(z)
\end{equation}
We calculate the $n$-th derivative of $y$  
$$
y^{(n)}=q_{0,n}(z)y+q_{1,n}(z)y',\ 
q_{i,n}\in\C(z),\ i=0,1.
$$
Theorem \ref{charla3} in this case is:
\begin{theo}
\label{charla2}
 For a generic Fuchsian differential equation (\ref{lequation}), the fraction 
 $\frac{q_{0,n}}{q_{1,n}}$ converges uniformly in compact subsets of $U_i$ to
  $-\frac{y_2'}{y_2}$,
   where $y_2$ is the unique (up to multiplication by a constant) holomorphic solution of (\ref{lequation}) in $U_i$.   
 \end{theo}
It is not hard to verify that $-\frac{q_{0,n}}{q_{1,n}}$ is equal to the continued fraction (\ref{cfr}), where $a_{i,n},\ i=0,1$ are given by (\ref{recrel}), see \S \ref{cfsection}.  

\begin{proof}[Proof of Theorem  \ref{charla2}]
The convergence  follows from Theorem \ref{charla3}.
For the case $m=2$ we continue the proof of Theorem \ref{charla3}. We have 
$$
[Y,Y']^{-1}=\frac{1}{y_1y_2'-y_2y_1'}\mat{y_2'}{-y_1'}{-y_2}{y_1}
$$
and so $\frac{q_{0,n}}{q_{1,n}}$ converges to $-\frac{y_2'}{y_2}$. In the next section we will see that $\frac{q_{0,n}}{q_{1,n}}$ can be written as the continued fraction (\ref{cfr}).
\end{proof}

\section{Continued fractions}
\label{cfsection}
Let us consider the differential equation (\ref{lequation})  and 
\begin{equation}
\label{recrel2}
x_n=a_{1,n}x_{n+1}+a_{2,n}x_{n+2}+\cdots+a_{m,n}x_{n+m}, 
 \text{  where } x_n=\frac{y^{(n)}}{n!},\ n=0,1,2,\ldots 
\end{equation}
The recursive relations (\ref{recrel2}) gives us a monster continued fraction as follows:
We write (\ref{recrel2}) in the following form:
$$
\frac{x_{n+1}}{x_n}=\frac{1}{a_{1,n}+a_{2,n}\frac{x_{n+2}}{x_{n+1}}+ a_{3,n}\frac{x_{n+3}}{x_{n+2}}\frac{x_{n+2}}{x_{n+1}}+  \cdots+a_{m,n} \Pi_{i=0}^{m-2} \frac{x_{n+m-i}}{x_{n+m-i-1}} }
$$
Now we take $\frac{x_1}{x_0}$ and write it in terms of $\frac{x_i}{x_{i-1}},\ i=2,\ldots,m$. In the next step we replace $\frac{x_2}{x_{1}}$, wherever it appears,  with the term given by the above formula for $n=1$. We repeat this procedure until infinity. The reader may have noticed that it is not possible to write down all these  substitutions in a A4 paper for arbitrary $m$. For $m=2$ we get the usual continued fraction and only in this case we are able to analyze its convergence 

Let us recall some notation from \cite{perron} \S 57 concerning continued fractions. Let $a_n,b_{n-1},\ n=1,2,\ldots$ be two sequences of complex numbers. We write
\begin{equation}
\frac{A_n}{B_n}=
b_0+\cfrac{a_1}{b_1+\cfrac{a_{2}}{\ddots\cfrac{\ddots}{b_{n-1}+\cfrac{a_n}{b_n}}}}
\end{equation}
and we have
$$
A_n=b_nA_{n-1}+a_nA_{n-2},\ B_n=b_nB_{n-1}+a_nB_{n-2}
$$
$$
A_{-1}=1,\ B_{-1}=0,\ A_0=b_0, \ B_0=1,\ 
$$
If there is a sequence of numbers $x_n,\ n=0,1,2\ldots,$ satisfying 
$$
x_n=b_nx_{n+1}+a_{n+1}x_{n+2},\ n=0,1,\ldots
$$
then for $n\geq 1$ we have
\begin{equation}
\label{x_0x_1}
x_0=A_{n-1}x_n+a_{n}A_{n-2}x_{n+1},
\end{equation}
$$
x_1=B_{n-1}x_n+a_{n}B_{n-2}x_{n+1},
$$

Now let us come back to our notations of linear differential equations of order two. The equalities (\ref{x_0x_1}) in our case implies that $-\frac{q_{0,n}}{q_{1,n}}$ is the continued fraction (\ref{cfr}). 

\begin{rem}\rm
After we wrote this note we became aware of  a paper by Norlund \cite{nor} which in page 447 essentially asserts a solution to Problem 1 in full generality in the case of second order Fuchsian differential equations. However his proof has nontrivial gaps. 
Let us explain one of the cases which he considers. Assuming that the indicial equation  of the second order linear differential equation $L$ at $z=t_i$ has two distinct roots $\alpha_j,\ j=1,2$, in a neighborhood of $t_i$ we have two solutions $f_j$ which are asymptotic to $y_j=(z-t_i)^{\alpha_j}$. He uses 
\begin{equation}
\label{15aug}
\lim_{n\to \infty}\frac{f_j^{(n)}(z)}{ y_j^{(n)}(z)}=1,\ \ z\in U_i
\end{equation}
and then he concludes that the limit (\ref{15aug2013}) is equal to $\lim_{n\to \infty}\frac{y_1^{(n)}(z)}{ y_2^{(n)}(z)}$. 
In general, his argument is that if $f_j$ near $t_i$ is asymptotic to $y_j$ then (\ref{15aug})
must hold. Note that for a Fuchsian differential equations we can take $y_j$ as polynomials in $\ln(z-t_i)$ and 
$(z-t_i)^{\alpha_i}$ and so calculating $y_j^{(n)}$ is easy and explicit.  It is not clear however why this kind of assertion must be true or, in other words, which kind of generic conditions we have to put on $L$ such that (\ref{15aug}) holds. For the asymptotic behavior of $f_j^{(n)}$ for growing $n$, Norlund refers to the work of Perron \cite{per1913}.

  \end{rem}

{}

\end{document}